\documentclass{amsart}
\usepackage{amsmath}

\newtheorem*{questionM}{\textbf{Problem C}}
\newtheorem*{questionA}{\textbf{Problem C in $\Z\{z\}$}}

\newtheorem{theorem}{\textbf{Theorem}}

\newtheorem{fact}{\textbf{Fact}}
\newtheorem{corollary}{\textbf{Corollary}}
\newtheorem{remark}{\textbf{Remark}}

\newtheorem{lemma}{\textbf{Lemma}}

\def\a {\alpha}
\def\aa {\overline{\alpha}}
\def\b {\beta}
\def\N {\mathbb{N}}
\def\Z {\mathbb{Z}}

\def\QQQ {\QQ\cap B(0,1)}
\def\Q {\mathbb{Q}}
\def\R {\mathbb{R}}
\def\ZZ {\Z\{z\}}

\def\g {\gamma}
\def\QQ {\overline{\Q}}
\def\C {\mathbb{C}}

\theoremstyle{remark}

\numberwithin{equation}{section}

\begin{document}

\title[A QUESTION PROPOSED BY K. MAHLER ON EXCEPTIONAL SETS]{ON EXCEPTIONAL SETS OF TRANSCENDENTAL FUNCTIONS WITH INTEGER COEFFICIENTS: SOLUTION OF A MAHLER'S PROBLEM}

\author{DIEGO MARQUES}
\address{DEPARTAMENTO DE MATEM\'{A}TICA, UNIVERSIDADE DE BRAS\'ILIA, BRAS\'ILIA, DF, BRAZIL}
\email{diego@mat.unb.br}


\author{CARLOS GUSTAVO MOREIRA}
\address{INSTITUTO DE MATEM\' ATICA PURA E APLICADA, RIO DE JANEIRO, RJ, BRAZIL}
\email{gugu@impa.br}

\subjclass[2010]{Primary 11Jxx, Secondary 30Dxx}

\keywords{Mahler problem, transcendental function, Liouville's inequality}

\begin{abstract}
In this paper, we shall prove that any subset of
$\overline{\Q}\cap B(0,1)$, which is closed under complex conjugation and which contains the element $0$, is the exceptional set of uncountably many transcendental functions, analytic in the unit ball, with integer coefficients. This solves a strong version of an old question proposed by K. Mahler (1976).
\end{abstract}

\maketitle

\section{Introduction}

A \textit{transcendental function} is a function $f(x)$ such that the only complex polynomial satisfying $P(x, f(x)) =0$, for all $x$ in its domain, is the null polynomial. For instance, the trigonometric functions, the exponential function, and their inverses. 

The most interesting classes of numbers for which transcendence has been proved are given as the values of suitable analytic transcendental functions. These functions, in many cases, are defined as power series with integral or rational or algebraic coefficients. Weierstrass initiated the question of investigating the set of algebraic
numbers where a given transcendental entire function $f$ takes algebraic values. Denote by $\QQ$ the field of algebraic numbers. For an entire function $f$, we define the exceptional set $S_f$ of $f$ as 
\[
S_f=\{\alpha\in \QQ\ :\ f(\alpha)\in \QQ\}.
\]

For instance, the Hermite-Lindemann theorem implies that if $S\subseteq \QQ$ is finite then the exceptional set of $\exp(\prod_{\alpha\in S}(z-\alpha))$ is $S$. The exceptional sets of the functions $2^z$ and $e^{z\pi +1}$ are $\Q$ and $\emptyset$, respectively, as shown by the Gelfond-Schneider theorem and Baker's theorem. Also, assuming Schanuel's conjecture, we obtain that the exceptional sets of $2^{2^z}$ and $2^{2^{2^{z-1}}}$ are $\Z$ and $\Z_{>0}$, respectively.

The study of exceptional sets started in 1886 with a letter from Weierstrass to
Strauss. In this letter, Weierstrass conjectured about the existence of a transcendental entire function whose exceptional set is $\QQ$. This assertion was proved in 1895 by St\"{a}ckel \cite{19} who established a much more general result: \textit{for each countable subset $\Sigma\subseteq \C$ and each dense subset $T\subseteq \C$, there exists a transcendental entire
function $f$ such that $f(\Sigma) \subseteq T$} (Weierstrass assertion is obtained by choosing $\Sigma=T=\QQ$). Since that time, many mathematicians have posed conjectures on values of transcendental functions at algebraic points. Surprisingly, most of these conjectures turned out to be wrong and they were studied by mathematicians like Weierstrass, Strauss, St\"{a}ckel, Faber, Hurwitz, Gelfond, Lekkerkerker, Mahler, Waldschmidt and many others. We still remark that the Hilbert's seventh problem include the following comment (as can be see in \cite[Chap. 6]{tubbs}):
\begin{center}
\begin{minipage}{10 cm}
\textit{We expect transcendental functions to assume, in general, transcendental values for $[\ldots]$ algebraic arguments.}
\end{minipage}

\end{center}

The question of the possible sets $S_f$ has been solved in \cite{diego}: \textit{any subset of algebraic numbers is the exceptional set of some transcendental entire function}. However, none information about the arithmetic nature of the coefficients of the Taylor series of $f$ is obtained in that construction.

We point out that in one of his books, Mahler \cite[Chap. 3]{bookmahler} investigated the possible exceptional sets of entire functions having rational coefficients in their Taylor series (the set of these functions is denoted by $T_{\infty}$). In particular, he claimed to have a proof for the following result: \textit{If $S$ is closed relative to $\QQ$} (i.e., if $\alpha\in S$, then any algebraic conjugate of $\alpha$ also lies in $S$) and $0\in S$, then there is a function $f\in T_{\infty}$ such that $S_f=S$. The interest in the arithmetic behavior of the coefficients of the Taylor series of the functions may have origin in the Straus-Schneider theorem which relates the cardinality of the set of algebraic values where a function take, together with all its derivatives, integer values with the order of the function (see \cite[Chap. 9]{chud} for more details on these kind of results).

In the same book, Mahler suggested three problems on the arithmetic behavior of transcendental functions. He named them as problems A, B and C. The problem C is exactly based on the previous cited work of Mahler. More precisely

\begin{questionM}
Does there exist for any choice of $S$ (closed under complex conjugation and such that $0\in S$) a series $f$ in $T_{\infty}$ for which $S_f=S$?
\end{questionM}

We remark that the sentences in parentheses do not appear in the Mahler's original question. However, they are necessary since for any function $f\in T_{\infty}$, one has that $\overline{f(\alpha)}=f(\overline{\alpha})$ and that $f(0)$ is a rational number. So, our guess is that in his statement, ``any choice"\ means ``any possible choice".

This question was answered positively by Marques and Ramirez \cite{MR} (we refer the reader to \cite{DG} for the solution of Problem B). Their constrution is simple and it is based on the fact that $\Q$ and $\Q(i)$ are dense in $\R$ and in $\C$, respectively.

Another kind of problem appears when we require the coefficients to be integers. It is almost unnecessary to stress that the question for integer coefficients is substantially harder than the rational case, since $\Z$ is not dense in $\R$. However, in 1965, Mahler \cite{M} already studied the arithmetic behavior of transcendental functions with integer coefficients. Indeed, he proved that every set $S\subseteq \QQ\cap B(0,1)$, which is closed relative to $\QQ$ and such that $0\in S$, is exceptional for some transcendental function in $\ZZ$ (here, as usual,  $\Z\{z\}$ denotes the set of the power series with integer coefficients and which are analytic in $B(0,1)$). Therefore, a question arise: how about the Mahler question for integer coefficients? That is,

\begin{questionA}
Does there exist for any choice of $S\subseteq \QQ\cap B(0,1)$ (closed under complex conjugation and such that $0\in S$) a transcendental function $f \in \Z\{z\}$ for which $S_f=S$?
\end{questionA}

We refer the reader to \cite{bookmahler,wal1} (and references therein) for more results about the arithmetic behavior of transcendental functions.

In this paper, we give an affirmative answer to the previous question which fullfil the work of Mahler about these functions. More precisely, we prove that
\begin{theorem}\label{1}
Every subset of $\QQ\cap B(0,1)$, closed under complex conjugation which contains the element $0$, is the exceptional set of uncountably many transcendental functions in $\Z\{z\}$.
\end{theorem}

Our proof combines some key lemmas, which by themselves can be of widely theoretical interest, with a Bombieri result which gives tools for proving transcendence.

\section{Key lemmas}

In this section, we shall provide some lemmas which are essential ingredients in our proof.

The first lemma is a particular case of a result due to Harbater \cite[Lemma 1.5]{H}, according to which, if $r$ is a positive real number less than $1$ and $\lambda$ is
a non-zero complex number of absolute value at most $r$, then there is a function
$f\in \Z\{t\}$ such that $f(0)=1$ which vanishes to order $1$ at $\lambda$ and its
complex conjugate, and vanishes nowhere else in the disc $|t|\le r$.

\begin{lemma}\label{l1}
Let $\a,\b\in B(0,1)$ with $\b\not\in \{\a,\aa\}$. Then there exists a function $f\in \Z\{z\}$ such that $f(\a)=0$ and $f(\b)\neq 0$. Moreover, there exists a positive constant $C$ depending on $\alpha$ and $\beta$ such that $|f(z)|\leq C/(1-|z|)$, for all $z\in B(0,1)$. 
\end{lemma}
\begin{proof}
Take $r=(\max\{|\a|,|\b|\}+1)/2$ and $\lambda=\alpha$ in the Lemma 1.5 of \cite{H}. That lemma ensures the existence of a function $f\in \ZZ$ such that the only zeros of $f$ inside $B(0,r)$ are $\alpha$ and $\overline{\a}$. Since $\b\in B(0,r)$, then $f(\b)\neq 0$. Also, Harbater construction gives a function of the form $f(z)=f_s(z)(1+b_1z+b_2z^2+\cdots)$, where $f_s(z)$ is a polynomial with coefficients depending only on $\a$ and $\overline{\a}$ and $|b_i|\leq 1/2$ (also $f_s(0)=1$). Therefore, the coefficients of $f$ are bounded by $L(f_s)/2$ (where $L(f)$ denotes the {\it length} of the polynomial $f$, i.e., the sum of the absolute values of its coefficients) and this length depends only on $\a, \overline{\a}$ and $s=\lfloor \log (2(1-r))/\log r\rfloor +1$. This gives our desired bound.
\end{proof}

For the next lemma, it is not possible to use Harbater's result directly, since the set of algebraic numbers has limit points at $|z|=1$. 

\begin{lemma}\label{l2}
Let $\a\in \QQQ$. Then there exists a function $f\in \ZZ$ such that $f(z)=0$ for $z\in \QQQ$ if and only if $z\in \{\a,\aa\}$. Moreover, there exists a positive constant $C$, such that $|f(z)|\leq C/ (1-|z|)^3$, for all $z\in B(0,1)$.
\end{lemma}
\begin{proof}
If $\a=0$, take $f(z)=z$, which satisfies $|f(z)|\leq 1/ (1-|z|)^3, \forall z\in B(0,1)$. So we may suppose $\a\ne 0$. Write $\QQQ\backslash\{\a,\aa\}=\{\beta_1,\beta_2,\ldots\}\cup \{\overline{\b}_1,\overline{\b}_2,\ldots\}$ with $\beta_1=0$ and $\beta_k\ne 0$ for $k>1$. By Lemma \ref{l1}, for all $j\geq 1$, there exists a function $f_j\in \Z\{z\}$ such that $f_j(\a)=0$ and $f_j(\b_j)\neq 0$, and there is a constant $C_j\ge 1$ such that $|f_j(z)|\leq C_j/(1-|z|)$, for all $z\in B(0,1)$. The function $f$ will be defined by $f(z):=\sum_{k\geq 1}z^{t_k}f_k(z)$, where $(t_k)_k$ is an increasing sequence of natural numbers which will be chosen later satisfying $t_1=0$ and $t_k\ge C_k+k, \forall k>1$. This implies that the function $f$ is analytic in $B(0,1)$ and satisfies $|f(z)|\leq C/ (1-|z|)^3, \forall z\in B(0,1)$, for some positive constant $C$. Indeed, for $z=0$, one has $|f(0)|=|f_1(0)|\le C:=\lfloor |f_1(0)| \rfloor+1$. Now, for $0<|z|<1$, the function $x\mapsto x|z|^x$ has maximum at $x=1/|\log |z||$ and so we have that $C_k|z|^{C_k}\leq e^{-1}/|\log |z||$. Thus, $|f(z)|\leq (e(1-|z|)|\log |z||)^{-1}\sum_{k\geq 0}|z|^{m_k}\leq (e(1-|z|)^2|\log|z||)^{-1}\leq 1/(1-|z|)^3$, where we used that $|\log |z||\geq 1-|z|$, for $z\in B(0,1)$. To finish we obtain our estimate by combining this case with the case $z=0$ (since $C\ge 1$).

We start with the sequence $({\hat t}_k)_k$ given by ${\hat t}_1=0$ and, for $k>1$, 
$${\hat t}_k=\max\{C_k+k,
{\hat t}_{k-1}+\frac{(k\log 2+\log C_k)+\max_{1<r\le k}|\log |f_r(\b_r)(1-|\b_r|)||}{\min_{1<r\le k}|\log |\b_r||}\}.$$
The sequence $({\hat t}_k)_k$ is increasing and satisfies that, for $k>n>1$, 
\begin{equation}\label{O}
{\hat t}_k-{\hat t}_n\ge {\hat t}_k-{\hat t}_{k-1}\ge \frac{(k\log 2+\log C_k)+|\log |f_n(\b_n)(1-|\b_n|)||}{|\log |\b_n||}.
\end{equation}
This implies that, for $k>n>1$, $|\b_n|^{{\hat t}_k-{\hat t}_n}C_k\le 2^{-k}|f_n(\b_n)|(1-|\b_n|)$ and thus
\[
\left|\displaystyle\sum_{k\geq n+1}\beta_n^{{\hat t}_k}f_k(\b_n)\right|\le |\b_n|^{{\hat t}_n}\sum_{k\geq n+1}2^{-k}|f_n(\b_n)|=2^{-n}|f_n(\b_n)||\b_n|^{{\hat t}_n}\le \frac14 |\b_n|^{{\hat t}_n}|f_n(\b_n)|
\]

Now, we modify (if necessary) the sequence $({\hat t}_k)_k$ in order to obtain a sequence $(t_k)_k$ such that, for any $N\ge 1$,
\begin{equation}\label{I}
\left|\displaystyle\sum_{k=1}^N\beta_N^{t_k}f_k(\b_N)\right|>\frac{1}{2}|\b_N|^{t_N}|f_N(\b_N)|.
\end{equation}
Set initially $t_n={\hat t}_n$ for every $n\ge 1$. We will modify inductively this sequence - in the $n$-th step we may modify its $n$-th term and the subsequent terms, but not the previous ones.

For $N=1$, the above inequality clearly holds. Suppose, by induction hypothesis, that (\ref{I}) holds for all $1\leq j<N$. Now, we need to choose $t_N$. If (\ref{I}) is valid for $t_N$, then we have nothing to do. So, supposing the contrary, we obtain
\begin{eqnarray*}
|\b_N|^{t_N}|f_N(\b_N)|-\left|\displaystyle\sum_{k=1}^{N-1}\beta_N^{t_k}f_k(\b_N)\right| & \leq & \left|\displaystyle\sum_{k=1}^{N-1}\beta_N^{t_k}f_k(\b_N)+ \b_N^{t_N}f_N(\b_N)\right|\\
 &\leq &  \frac{1}{2}|\b_N|^{t_N}|f_N(\b_N)|.
\end{eqnarray*}
This implies that
\begin{equation}\label{II}
\left|\displaystyle\sum_{k=1}^{N-1}\beta_N^{t_k}f_k(\b_N)\right|\geq  \frac{1}{2}|\b_N|^{t_N}|f_N(\b_N)|.
\end{equation}
Let $\ell$ be the smallest positive integer such that $|\b_N|^{\ell}<1/4$. We shall prove that, if we replace $t_N$ by $\tilde{t}_N=t_N+\ell$, then (\ref{I}) holds. In fact, first note that $4|\b_N|^{\tilde{t}_N}|f_N(\b_N)|<|\b_N^{t_N}||f_N{\b_N}|$ and then
\begin{eqnarray*}
\left|\displaystyle\sum_{k=1}^{N-1}\beta_N^{t_k}f_k(\b_N)+ \b_N^{\tilde{t}_N}f_N(\b_N)\right| & \geq & \left|\displaystyle\sum_{k=1}^{N-1}\beta_N^{t_k}f_k(\b_N)\right|-|\b_N|^{\tilde{t}_N}|f_N(\b_N)|\\
& > & |\b_N|^{\tilde{t}_N}|f_N(\b_N)|
\end{eqnarray*}
as desired. This proves the inequality in (\ref{I}). 

Now, replace $t_m$ by $t_m+\ell$ for every $m\ge N$, and repeat this process inductively, replacing $N$ by $N+1$ in the above construction.

Notice that, in the end, we get a sequence $(t_k)_k$ which still satisfies (\ref{O}), and so, for every $N>1$, 
$$\left|\sum_{k\geq N+1}\beta_N^{t_k}f_k(\b_N)\right|\le\frac14 |\b_N|^{t_N}|f_N(\b_N)|$$

Thus, we have by inequality in (\ref{I}), 
\begin{eqnarray*}
|f(\b_N)| & \geq & \left|\displaystyle\sum_{k=1}^N\beta_N^{t_k}f_k(\b_N)\right|-\left|\displaystyle\sum_{k\geq N+1}\beta_N^{t_k}f_k(\b_N)\right|\\
& \geq & \frac{1}{2}|\b_N|^{t_N}|f_N(\b_N)|-\frac14 |\b_N|^{t_N}|f_N(\b_N)|=\frac14 |\b_N|^{t_N}|f_N(\b_N)|>0. 
\end{eqnarray*}
This implies that $f(\b_N)\ne 0$ for every $N>1$. Since $f(\b_1)=f(0)=f_1(0)=f_1(\b_1)\ne 0$, the proof is complete.
\end{proof}

\begin{lemma}\label{l3}
Let $\a\in B(0,1)$ and $\b\in \C$ such that $\a\neq 0$ and if $\a$ is real then so is $\b$. Then there exists a function $g\in \ZZ$, with bounded coefficients, such that $g(\a)=\b$. 
\end{lemma}
\begin{proof}
Suppose that $\a$ and then $\b$ are real numbers. We may also assume that $\a>0$ (otherwise, we replace $\a$ by $\a^2$). Hence we can write the $1/\a$-expansion of $\b$ as 
\[
\b=a_0+\frac{a_1}{1/\a}+\frac{a_2}{(1/\a)^2}+\frac{a_3}{(1/\a)^3}+\cdots,
\]
where $a_0,a_1,\ldots$ are integers with $a_0=\lfloor \b\rfloor$ and $0\leq a_k\leq 1/\a$. Therefore the desired function has the form $g(z)=\sum_{k\geq 0}a_kz^k$. 

Now, let us suppose that $\a\notin \R$. In this case, we shall need the following key fact: 

\begin{fact}
Let~$\alpha\not\in \R$. Then there exists ${K=K(\alpha)>0}$ such that the following holds.  For any ${z\in \C}$ there exist unique ${x,y\in \R}$ such that ${z=x+y\alpha}$ and 
${|x|,|y|\le  K|z|}$.
\end{fact}
\begin{proof}
Since ${\alpha \notin\R}$, any ${z\in \C}$ can be uniquely written as ${z=x+y\alpha}$ with ${x,y\in \R}$. Then 
${z\mapsto \max\{|x|,|y|\}}$ defines an $\R$-norm on~$\C$. Since any two norms on~$\C$ are equivalent, the result follows.  
\end{proof}

In order to prove the lemma, we will define two bounded sequences of integers  $(b_k)_{k\geq 0}$ and $(c_k)_{k\geq 0}$ such that 
\begin{equation}
\label{etwosums}
\sum_{k=0}^\infty b_k\alpha^{2k}+\sum_{k=0}^\infty c_k\alpha^{2k+1}=\beta.
\end{equation}
Write ${\beta=x+y\alpha}$ with ${x,y\in \R}$ and set 
\[
b_0=\lfloor x\rfloor, \qquad c_0=\lfloor y\rfloor.
\]
By the previous fact, we have 
\begin{equation}
\label{ezero}
|b_0|, |c_0|\le K|\beta|+1. 
\end{equation}
and we have clearly 
\[
|\beta -(b_0+c_0\alpha)|\le 1+|\alpha|\le 2. 
\]
Now assume that ${b_0, \ldots, b_n, c_0, \ldots, c_n\in \Z}$ are defined to satisfy~\eqref{ezero} and the following two conditions:
\begin{align}
\label{ek}
&|b_k|, |c_k|\le 2K|\alpha|^{-2}+1 \qquad (k=1, \ldots, n),\\
\label{erest}
&\left|\beta -\left(\sum_{k=0}^n b_k\alpha^{2k}+\sum_{k=0}^n c_k\alpha^{2k+1}\right)\right| \le 2|\alpha|^{2n}.
\end{align}
We will define ${b_{n+1}, c_{n+1}\in \Z}$ so that~\eqref{ek} would also hold for ${k=n+1}$, and~\eqref{erest} would hold with~$n$ replaced by ${n+1}$.

Set
\[
z=\frac{\beta -\left(\sum_{k=0}^n b_k\alpha^{2k}+\sum_{k=0}^n c_k\alpha^{2k+1}\right)}{\alpha^{2n+2}}. 
\]
Then ${|z|\le 2|\alpha|^{-2}}$ by \eqref{erest}. Applying the fact, we find $ {x,y\in \R}$ such that ${z=x+y\alpha}$ and 
\[
|x|,|y|\le K|z|\le 2K|\alpha|^{-2}.
\]
Setting 
\[
b_{n+1}=\lfloor x\rfloor, \qquad c_{n+1}=\lfloor y\rfloor,
\]
we have~\eqref{ek} with ${k=n+1}$. Further, we have clearly 
\[
|z -(b_{n+1}+c_{n+1}\alpha)|\le 1+|\alpha|\le 2. 
\]
With our definition of~$z$ this re-writes as 
\[
\left|\frac{\beta -\left(\sum_{k=0}^{n+1} b_k\alpha^{2k}+\sum_{k=0}^{n+1} c_k\alpha^{2k+1}\right)}{\alpha^{2n+2}}\right| \le 2, 
\]
which is~\eqref{erest}  with~$n$ replaced by ${n+1}$.

Thus, we defined bounded sequences of integers  $(b_k)_{k\geq 0}$ and $(c_k)_{k\geq 0}$ such that~\eqref{erest} holds for all~$n$. Hence~\eqref{etwosums} holds. In conclusion, the lemma follows for the choice of $g(z)=\sum_{k\geq 0}a_kz^k$, where
$$
a_n=
\begin{cases}
b_m, &n=2m,\\
c_m,&n=2m+1.
\end{cases}
$$ 
\end{proof}

\begin{remark}
We can see in the proof of the previous lemma that if $\beta\in B(0,1)$ then the upper bound for the coefficients of $g$ will depend only on $\a$. We shall use this in the next construction.
\end{remark}

\begin{lemma}\label{l4}
Let $\a\in \QQQ$. Then there exists a function $f\in \ZZ$ such that $f(\QQQ)\subseteq \Q(i)$ and $f(z)=0$ for $z\in \QQQ$ if and only if $z\in \{\a,\aa\}$. Moreover, $|f(z)|\leq Ce^{2/(1-|z|)^2}$, for all $z\in B(0,1)$, where $C>0$ depends only on $\a$.
\end{lemma}
\begin{proof}
Write $\QQQ\backslash\{\a,\aa\}=\{\beta_1,\beta_2,\ldots\}\cup \{\overline{\b}_1,\overline{\b}_2,\ldots\}$. By Lemma \ref{l2}, there exist $f_0,f_1,\ldots$ with integer coefficients such that $f_0(z)=0$ for $z\in \QQQ$ if and only if $z\in \{\a,\aa\}$ and $f_j(z)=0$ for $z\in \QQQ$ if and only if $z\in \{\b_j,\overline{\b}_j\}$, for all $j\geq 1$. Now, define $f(z):=\sum_{k\geq 0}z^{t_k}g_k(z)f_0(z)\cdots f_k(z)$, where $(t_k)_k$ will be choose later. The functions $g_k$'s will be chosen inductively according with Lemma \ref{l3} and we will have $|g_k(z)|\leq d_k/(1-|z|)$ (where $d_k$ depends only on $\b_{k+1}$). By the bounds $|f_j(z)|\leq C_j/(1-|z|)^3$ given in the previous lemma, it is enough to choose $t_k=3k^2+\lceil (1-\delta_{0,k})d_kC_0\cdots C_k\rceil $ in order to obtain the analiticity of $f$ in $B(0,1)$. Note that we have $f(\a)=0$. Also, $f(\b_1)=g_0(\b_1)f_0(\b_1)$ and $f_0(\b_1)\neq 0$. Therefore, for some $\gamma\in B(0,1)$, we have that $\gamma f_0(\b_1)\in \Q(i)^*$. Now, by Lemma \ref{l3}, there exists a function $g_0\in \Z\{z\}$ such that $g_0(\b_1)=\gamma$. Thus, for this choice of $g_0$, one has that $f(\b_1)$ is a nonzero Gaussian rational. Assume, from now on, that $\b_j\neq 0$, for $j\geq 2$ and suppose that $g_0,\ldots, g_{N-1}$ were chosen such that $f(\b_j)\in \Q(i)^*$ for all $1\leq j\leq N$. Then,
\begin{eqnarray*}
f(\b_{N+1}) & = & \displaystyle\sum_{k=0}^{N-1}\b_{N+1}^{t_k}g_k(\b_{N+1})f_0(\b_{N+1})\cdots f_k(\b_{N+1})\\
&  & +  \b_{N+1}^{t_N}g_N(\b_{N+1})f_0(\b_{N+1})\cdots f_N(\b_{N+1})\\
&= & \theta_N+\b_{N+1}^{t_N}g_N(\b_{N+1})f_0(\b_{N+1})\cdots f_N(\b_{N+1}).
\end{eqnarray*}
Since $\b_{N+1}^{t_N}f_0(\b_{N+1})\cdots f_N(\b_{N+1})\neq 0$, then there exists a complex number $\theta\in B(0,1)$ such that $\theta_N+\theta\b_{N+1}^{t_N}f_0(\b_{N+1})\cdots f_N(\b_{N+1})$ belongs to $\Q(i)^*$. Then, we choose $g_{N}\in \Z\{z\}$ (by Lemma \ref{l3}) such that $g_N(\b_{N+1})=\theta$. Thus, $f(\b_{N+1})\in \Q(i)^*$ and the construction is complete. 

In order to get the upper bound for $|f(z)|$, note that by the construction and the bounds in the previous lemmas, we have
\[
|f(z)|\leq \frac{D_0}{(1-|z|)^4}+\displaystyle\sum_{k\geq 1}D_k\frac{|z|^{t_k}}{(1-|z|)^{3k+4}},
\]
where $D_k=d_kC_0\cdots C_k$. Since $t_k=\lceil (1-\delta_{0,k})D_k\rceil +3k^2$, we have
\begin{eqnarray*}
|f(z)| &\leq & \frac{D_0}{(1-|z|)^4}+\frac1{(1-|z|)^4}\displaystyle\sum_{k\geq 1}D_k|z|^{D_k}\left(\frac{|z|^{k}}{1-|z|}\right)^{3k}\\
&\leq &  \frac{D_0}{(1-|z|)^4}+\frac1{(1-|z|)^5}\displaystyle\sum_{k\geq 0}\left(\frac{|z|^{k}}{1-|z|}\right)^{3k},
\end{eqnarray*}
since, as in the previous lemma, $D_k|z|^{D_k}\le 1/|\log |z||\le 1/(1-|z|)$. Also, $|z|^k<(1-|z|)/2$ when $k>(\log (1-|z|)-\log 2)/\log |z|=:k(z)$. Thus, 
 
\begin{eqnarray*}
|f(z)|  & \leq & \frac{D_0}{(1-|z|)^4}+\frac{1}{(1-|z|)^5}\left(\displaystyle\sum_{k= 1}^{\lfloor k(z)\rfloor}\left(\frac{|z|^{k}}{1-|z|}\right)^{3k}+\displaystyle\sum_{k\geq k(z)}\left(\frac{|z|^{k}}{1-|z|}\right)^{3k}\right)\\
 & \leq & \frac{D_0}{(1-|z|)^4}+\frac{1}{(1-|z|)^5}\left(\displaystyle\sum_{k= 1}^{\lfloor k(z)\rfloor}\left(\frac{|z|^k}{1-|z|}\right)^{3k}+\displaystyle\sum_{k>k(z)}\left(\frac{1}{2}\right)^{3k}\right)\\
 & \leq & \frac{D_0}{(1-|z|)^4}+\frac{1}{(1-|z|)^5}\left(\displaystyle\sum_{k= 1}^{\lfloor k(z)\rfloor}\left(\frac{|z|^k}{1-|z|}\right)^{3k}+\frac14\right).
\end{eqnarray*}
for all $z\in B(0,1)$. Note now that the maximum value of the function $t(k)=(\frac{|z|^k}{1-|z|})^{3k}$ for real positive values of $k$ is attained at $k=\frac{\log(1-|z|)}{2\log |z|}$, and is equal to $(\frac1{1-|z|})^{3\frac{\log(1-|z|)}{4\log |z|}}=e^{-3\frac{\log^2(1-|z|)}{4\log |z|}}$. This implies that
\[
|f(z)|\le \frac{D_0}{(1-|z|)^4}+\frac{1}{(1-|z|)^5}\left(\lfloor k(z)\rfloor e^{-3\frac{\log^2(1-|z|)}{4\log |z|}}+\frac14\right).
\]
Since $\lfloor k(z)\rfloor \le k(z)=(\log (1-|z|)-\log 2)/\log |z|\le (-\log (1-|z|)+\log 2)/(1-|z|)$, we get
\begin{eqnarray*}
|f(z)| &\le & \frac{D_0}{(1-|z|)^4}+\frac{1}{(1-|z|)^5}\left(\frac{-\log (1-|z|)+\log 2}{1-|z|} e^{-3\frac{\log^2(1-|z|)}{4\log |z|}}+\frac14\right)\\
&\le & \frac{D_0}{(1-|z|)^4}+\frac{1}{(1-|z|)^5}\left(\frac{-\log (1-|z|)+1}{1-|z|} e^{-3\frac{\log^2(1-|z|)}{4\log z}}\right) \\
& = & \frac{D_0}{(1-|z|)^4}+\frac{1}{(1-|z|)^6}\left((-\log (1-|z|)+1) e^{-3\frac{\log^2(1-|z|)}{4\log |z|}}\right)\\
& < & \frac{D_0}{(1-|z|)^4}+\frac{1}{(1-|z|)^7}e^{-3\frac{\log^2(1-|z|)}{4\log |z|}}\leq \frac{D_0}{(1-|z|)^4}+\frac{1}{(1-|z|)^7}e^{3\frac{\log^2(1-|z|)}{4(1-|z|)}}.
\end{eqnarray*}
Note now that the maximum of $x\log^2 x$ for $0<x<1$ is attained at $x=e^{-2}$, and is equal to $4/e^2<4/7$, so $|f(z)|\le \frac{D_0}{(1-|z|)^4}+\frac{1}{(1-|z|)^7}e^{\frac3{7(1-|z|)^2}}$. Also, the maximum of $x/e^{\alpha x^2}$ for $x\ge 1$ is attained at $x=1/\sqrt{2\alpha}$, and is equal to $1/\sqrt{2\alpha e}$, which is equal to $1$ when $\alpha=1/2e$. So we have $x\le e^{x^2/2e}$ for every $x\ge 1$, and thus $x^7\le e^{7x^2/2e},$ for all $x\ge 1$. This implies that
\[
|f(z)|\le \frac{D_0}{(1-|z|)^4}+e^{\frac{7/2e+3/7}{(1-|z|)^2}}<\frac{D_0}{(1-|z|)^4}+e^{\frac2{(1-|z|)^2}}<Ce^{\frac2{(1-|z|)^2}}
\]
(since $e^{2x}-e^{7x/4}=e^{7x/4}(e^{x/4}-1)\ge e^{7x/4}\cdot x/4\ge e^{7/4}/4>1$ for $x\ge 1$). Here $C=D_0+1$ and the result follows.
\end{proof}

The next lemma is of theoretical interest and can be found in \cite[p. 35]{bookmahler}. However, we shall provide a proof here (which is simpler than the Mahler's one) for the convenience of the reader.

\begin{lemma}\label{l5}
Let $g(x)=\sum_{k=0}^{\infty}a_k x^k$ be a power series with positive radius of convergence and rational coefficients. If there is a non-zero polynomial $f(x,y)\in \C[x,y]$ of degree $n$ such that $f(x,g(x))$ is identically $0$ then there is a non-zero polynomial $\tilde f(x,y)\in \Q[x,y]$ of degree at most $n$ such that $\tilde f(x,g(x))$ is identically $0$.
\end{lemma}

\begin{proof}
Since there is a non-zero polynomial $f(x,y)\in \C[x,y]$ of degree $n$ such that $f(x,g(x))\equiv 0$, the power series $x^r g(x)^s, r, s\in \N, r+s\le n$ are linearly dependent. For every $N\in \N$, let $\pi_N:\C[[x]]\to \C[x]$ be the natural projection on the vector space of polynomials with degree at most $N$ given by $\pi_N(\sum_{k=0}^{\infty}a_k x^k)=\sum_{k=0}^N a_k x^k$. Let $E<\C[[x]]$ be the vector space generated by $x^r g(x)^s, r, s\in \N, r+s\le n$ and, for each $N\in \N$, $E_N:=\pi_N(E)$. We have $0\le \dim E_N\le \dim E\le (n+1)(n+2)/2$ and $\dim E_N \le \dim E_{N+1}$, for every $N\in \N$.
So there are $d, N_0\in \N$ such that $\dim E_N=d$, for all $N\ge N_0$. Let $(r_1,s_1),\dots,(r_d,s_d)$ be such that $r_i, s_i\in \N$, $r_i+s_i\le n,$ for all $i\le d$ and $\pi_{N_0}(x^{r_i} g(x)^{s_i}), 1\le i\le d$ form a basis of $E_{N_0}$. We claim that $\pi_N(x^{r_i} g(x)^{s_i}), 1\le i\le d$ form a basis of $E_N$ for every $N\ge N_0$. Indeed they are linearly independent (since otherwise $\pi_{N_0}(x^{r_i} g(x)^{s_i})=\pi_{N_0}(\pi_N(x^{r_i} g(x)^{s_i})), 1\le i\le d$ would be linearly dependent, a contradiction) and $\dim E_N=d$. 

Since $x^r g(x)^s, r, s\in \N, r+s\le n$ are linearly dependent in $\C[[x]]$, applying $\pi_N$ to a non-trivial linear combination equal to $0$, we conclude that for any $N\in \N$, $\pi_N(x^r g(x)^s), r, s\in \N, r+s\le n$ are linearly dependent in $\C[x]$, so $d<(n+1)(n+2)/2$ and there is $(r,s)$ with $r, s\in \N, r+s\le n$ and $(r,s)\ne (r_i,s_i)$ for $1\le i\le d$. Since $\pi_{N_0}(x^{r_i} g(x)^{s_i}), 1\le i\le d$ form a basis of $E_{N_0}$, and $\pi_N(x^r g(x)^s)\in \Q[x]$, for every $N\in \N, r, s\in \N, r+s\le n$, it follows that there are $c_i\in \Q, 1\le i\le d$, uniquely determined, such that $\pi_{N_0}(x^r g(x)^s)=\sum_{i=1}^d c_i \pi_{N_0}(x^{r_i} g(x)^{s_i})$. This implies that, for every $N\ge N_0$, $\pi_N(x^r g(x)^s)=\sum_{i=1}^d c_i \pi_N(x^{r_i} g(x)^{s_i})$. Indeed, given $N\ge N_0$, there are $\tilde c_i\in \Q, 1\le i\le d$ such that $\pi_N(x^r g(x)^s)=\sum_{i=1}^d \tilde c_i \pi_N(x^{r_i} g(x)^{s_i})$ (since $\pi_N(x^{r_i} g(x)^{s_i}), 1\le i\le d$ form a basis of $E_N$); applying $\pi_{N_0}$, we get $\pi_{N_0}(x^r g(x)^s)=\sum_{i=1}^d \tilde c_i \pi_{N_0}(x^{r_i} g(x)^{s_i})$, and so, by uniqueness, $\tilde c_i=c_i$ for $1\le i\le d$. Now, since $\pi_N(x^r g(x)^s)=\sum_{i=1}^d c_i \pi_N(x^{r_i} g(x)^{s_i})$ for every $N\ge N_0$, it follows that $x^r g(x)^s=\sum_{i=1}^d c_i x^{r_i} g(x)^{s_i}$ in $\C[[x]]$, and so the conclusion holds with $\tilde f(x,y)=x^r y^s-\sum_{i=1}^d c_i x^{r_i} y^{s_i}\in \Q[x,y]$.
\end{proof}

\begin{corollary}\label{cor1}
The set of power series with positive radius of convergence and rational coefficients which are algebraic functions (i.e. such that there is a non-zero polynomial $f(x,y)\in \C[x,y]$ with $f(x,g(x))\equiv 0$) is countable.
\end{corollary}

Indeed, by the previous lemma, given such an algebraic power series, there is a non-zero polynomial $\tilde f(x,y)\in \Q[x,y]$ with $\tilde f(x,g(x))\equiv 0$, which implies the corollary since $\Q[x,y]$ is countable and any algebraic curve defined by a polynomial $\tilde f(x,y)\in \C[x,y]\setminus\{0\}$ has only a finite number of branches at $x=0$, i.e., there are only a finite number of power series $g(x)=\sum_{k=0}^{\infty}a_k x^k$ with positive radius of convergence (and complex coefficients) such that $\tilde f(x,g(x))\equiv 0$.

\section{The proof of the theorem}
Let $S=\{\a_1,\a_2,\ldots\}\cup \{\overline{\a_1},\overline{\a_2},\ldots\}\subseteq \QQ\cap B(0,1)$ and let $\QQ\cap B(0,1)\backslash S=\{\b_1,\b_2,\ldots\}$. For all $j\geq 1$, let $f_j$ as in Lemma \ref{l4}, i.e., $f_j(\QQQ)\subseteq \Q(i)$ and $f_j(z)=0$ for $z\in \QQQ$ if and only if $z\in \{\a_j,\aa_j\}$. Also, there exists a constant $C_j$ (depending on $\a_j$) such that $|f_j(z)|\leq C_je^{\frac2{(1-|z|)^2}}$. Set $e_k:=\lceil C_1\cdots C_k\rceil$. Define
\[
x_n:=\max_{1\leq i\leq n}\max_{1\leq k\leq n}\{H(f_1(\b_i)\cdots f_k(\b_i))\},
\]
where, as usual, the \textit{height} of an algebraic number $\b$, $H(\b)$, denotes the maximum of the absolute values of the coefficients of the minimal polynomial (over $\Z$) of $\b$ (If $\QQ\cap B(0,1)\backslash S=\emptyset$, then define $x_n=1,\ \forall n$, also if $\QQ\cap B(0,1)\backslash S$ is finite with $k$ elements, then define $\b_j=\b_k,\ \forall j>k$. Analogously, if $S$ is a finite set with $n$ elements, we can take $f_j(z)=f_1(z)$, for all $j>n$). 

Now, we define the sequence $(s_n)_n$ by $s_1=1$ and with two possible choices for $s_{n+1}$, namely, $s_{n+1}\in \{v_n, v_n+1\}$, where $v_n=n^n\max\{s_n,\lceil \log x_n\rceil, e_n\}$, for $n>1$. We claim that the function $f(z):=\sum_{k\geq 1}z^{s_k}f_1(z)\cdots f_k(z)$ belongs to $\Z\{z\}$ and has $S$ as its exceptional set. First, note that this function is analytic since for all $z$ belonging to the open ball $B(0,R)$, with $0<R<1$, one has that (by the estimates for $|f_j(z)|$)
\begin{eqnarray*}
|f(z)|& \leq & \displaystyle\sum_{k\geq 1} |z|^{s_k} e_ke^{2k/(1-|z|)^2}< \displaystyle\sum_{k\geq 1} R^{s_k} e_ke^{2k/(1-R)^2}.
\end{eqnarray*}
Since $s_k>k+e_k$ for all $k$ sufficiently large, we have $(R^{k+e_k} e_ke^{2/(1-R)^2})^k<1/2^k$, for all $k$ sufficiently large yielding that $f$ is an analytic function in $B(0,1)$, since the series which defines $f(z)$ converges uniformly in any of these balls.

For proving that $S_f=S$, first, let $\a\in S$, then $\a=\a_j$, for some $j$ and then $f(\a)=f(\a_j)=\sum_{k=1}^{j-1}\a_j^{s_k}f_1(\a_j)\cdots f_k(\a_j)\in \Q(\a_j,i)\subseteq \QQ$. Thus $\a\in S_f$. If $S=\QQ\cap B(0,1)$ the proof is complete. Otherwise, it remains us to prove that $f(\b)$ is a transcendental number, for all $\b\in \QQ\cap B(0,1)\backslash S$ . Suppose, towards a contradiction, that $f(\b)$ is an algebraic number of degree $t$ and that $\b=\b_j$ has degree $r$. Define $\gamma_N:=\sum_{k= 1}^{N}\b^{s_k}f_1(\b)\cdots f_k(\b)$. Note that $\g_N$ is an algebraic number of degree at most $2r$. Also, $\g_N\neq \g_{N+1}$ for all $N\geq 1$ (otherwise, $\b^{s_N}f_1(\b)\cdots f_N(\b)$ would be zero which contradicts the choice of $f_j$'s). Then, $f(\b)\neq \g_N$, for infinitely many integers $N$. Now, we shall use the following kind of {\it Liouville's inequality} given by Bombieri \cite{bomb}:
\begin{lemma}\label{bomb}
Let $\alpha_1,\a_2$ be distinct algebraic numbers of degrees $n_1$ and $n_2$, respectively. Then
\[
|\a_1-\a_2|>(4n_1n_2)^{-3n_1n_2}H(\a_1)^{-n_2}H(\a_2)^{-n_1}.
\]
\end{lemma}
We point out the existence of more accurated results about this kind of inequality (see, for instance, \cite[Hilfssatz 5]{G} and \cite[Corollary A.2]{bugeaud}), but we choose this one since it is enough for our purposes (also the implied contants in Bombieri's result are simpler).

For our case, take $N>j$, so, we have that
\begin{equation}\label{a1}
|f(\b)-\g_N|\gg H(f(\b))^{-2r}H(\g_N)^{-t}\gg H(\g_N)^{-t},
\end{equation}
where the implied constant depends only in $r, t$ and $f(\b)$.

On the other hand, 
\[
f(\b)-\g_N=\sum_{k\geq N+1}\b^{s_k}f_1(\b)\cdots f_k(\b)=\sum_{k\geq N+1}a_k\b^{s_k}.
\]
Now, let $c$ be a real number such that $0<|\b|<1/c<1$. Since the function $\sum_{k\geq 1}a_kz^{s_k}$ is analytic in the unit ball, then there exists a constant $c_2$ such that $|a_k(1/c)^{s_k}|\leq c_2$ and so $|a_k|\leq c_2c^{s_k}$, for all $k\geq 1$. Thus, 
\begin{equation}\label{a2}
|f(\b)-\g_N|\leq \sum_{k\geq N+1}|a_k||\b|^{s_k}\leq \sum_{k\geq N+1}c_2|c\b|^{s_k}\ll |c\b|^{s_{N+1}},
\end{equation}
where the implied constant does not depend on $N$.

Now, we need an upper bound for $H(\g_N)$. For that we need explicit upper bounds for $H(y^n), H(xy)$ and $H(y_1+\cdots+y_k)$ (these inequalities are known, but let us derive them here for the sake of completeness). Let $\a$ be a $d$-degree algebraic number, then there exists the relation 
\[
\frac{1}{d}\log H(\a)-\log 2\leq h(\a)\leq \frac{1}{d}\log H(\a)+\frac{1}{2d}\log (d+1),
\]
(this inequality can be seen in \cite[Lemma 3.11]{wal2}) where the \textit{logarithmic height} of an $s$-degree algebraic number $\g$ is defined as
\[
h(\g)=\dfrac{1}{s}(\log |a|+\displaystyle\sum_{j=1}^s\log \max\{1,|\g^{(j)}|\}),
\]
where $a$ is the leading coefficient of the minimal polynomial of $\gamma$ (over $\mathbb{Z}$) and $(\g^{(j)})_{1\leq j\leq s}$ are the conjugates of $\g$ (over $\Q$). 

Let $y_1,\ldots, y_k$ be algebraic numbers of degree $m_1,\ldots, m_k$, respectively. Since it is well-known that $h(y_1^n)=nh(y_1)$ (for a positive integer $n$), $ h(y_1y_2)\leq h(y_1)+h(y_2)$ and $h(y_1+y_2)\leq h(y_1)+h(y_2)+\log 2$ (see \cite[p. 75]{wal2}), then we got the following three relations
\begin{itemize}
\item $H(y_1^n)\leq e^{O(n)}H(y_1)^{n}$;
\item $H(y_1y_2)\leq e^{O(1)}(H(y_1)H(y_2))^{m_1m_2}$;
\item $H(y_1+\cdots +y_k)\leq e^{O(k)}(H(y_1)\cdots H(y_k))^{m_1\cdots m_k}$.
\end{itemize}
Here, the implied constants depend only on the degree of the algebraic numbers.

We point out that it is possible to obtain better upper bounds (see for example lemmas A.3 and A.4 in \cite{bugeaud}), but these ones are enough for our proof.

Since $\g_N$ is a finite sum, we can apply these inequalities and after some calculations, we get
\begin{eqnarray*}
H(\g_N) &\leq & e^{O(N)}(H(\b^{s_1}f_1(\b))\cdots H(\b^{s_N}f_1(\b)\cdots f_N(\b)))^{(2r)^N}\\
&\leq & e^{O(N)}(e^{O(N)}H(\b^{s_1})\cdots H(\b^{s_N})x_N^{N})^{(2r)^{N+1}}\\
&\leq &e^{O(N s_N(2r)^{N+1})}x_N^{N\cdot (2r)^{N+1}}
\end{eqnarray*}
for a sufficiently large $N>j$. Then, by (\ref{a1}), we have
\begin{equation}\label{a3}
|f(\b)-\g_N|\gg e^{-O(tNs_N(2r)^N)}x_N^{-tN(2r)^{N+1}}.
\end{equation}
By combining (\ref{a2}) and (\ref{a3}), we obtain
\[
|c\b|^{s_{N+1}}\gg e^{-O(tN s_N(2r)^{N+1})}x_N^{-Nt(2r)^{N+1}}.
\]
After some calculations, we arrive at
\[
\dfrac{s_{N+1}}{N(2r)^{N+1}\max\{s_N,\log x_N\}}\ll\dfrac{-t}{\log |c\b|}.
\]
However, by the definition of $s_{N+1}$, the left-hand side above tends to infinity as $N\to \infty$, contradicting the inequality. In conclusion, $f(\b)$ is a transcendental number and so $S_f=S$ as desired.

The proof that we can choose $f$ to be transcendental follows because there is a
binary tree of different possibilities for $f$. In fact, if we have chosen $s_1,\ldots, s_{n-1}$, then different
choices of $s_n$ (we have two choices in each step) give different values of $f(\alpha_{n+1})$, which does not depend on the values
of $s_k$, for $k > n$, and so different functions $f$. Thus, we have constructed uncountably
many possible functions, and by the Corollary \ref{cor1}, the set of the algebraic functions $f\in \Z\{z\}$ is a countable set. The proof is then complete.
\qed

\begin{remark}
Note that by using the same ideas as in the proof of Lemma \ref{l4} and in the proof of the transcendence of functions in the theorem, we obtain the following version of St\"{a}ckel's theorem for functions in $\Z\{z\}$ (we point out that there is no any information about the coefficients of the functions in the St\"{a}ckel original construction):
let $\Sigma\subseteq B(0,1)$ be a countable set and let $T$ be a dense subset of $\C$. Suppose that if $0\in \Sigma$, then $T\cap \Z\neq \emptyset$ (this condition is only because $f(0)$ must be an integer number). Then there exists a transcendental
function $f\in \Z\{z\}$ such that $f(\Sigma) \subseteq T$
\end{remark}


\section*{Acknowledgement}
The authors were supported by CNPq-Brazil. Part of this work was done during two very enjoyable visits of the first author to Universit\' e de Bordeaux and to IMPA (Rio de Janeiro). He thanks Yuri Bilu for nice discussions on this problem which lead to the proof of Lemma \ref{l3}. Also, he thanks Manjul Bhargava for discussions on some Mahler's problems during the Heildelberg Laureate Forum 2014. 





\begin{thebibliography}{9999}

\bibitem{bomb} E. Bombieri, Sull'approssimazione di numeri algebrici mediante numeri algebrici, \textit{Boll. Un. Mat. Ital.} \textbf{13} (1958), 351--354.

\bibitem{bugeaud} Bugeaud,Y.: \textit{Approximation by Algebraic Numbers}, Cambridge Tracts in Mathematics Vol {\bf 160}), Cambridge University Press, New York (2004).

\bibitem{chud} G. V. Chudnovsky, \textit{Contributions to the theory of Transcendental Numbers}, Mathematical Surveys and Monographs \textbf{19}, Amer. Math. Soc., 1984.

\bibitem{G} R. G\"{u}ting. Zur Berechnung der Mahlerschen Funktionen $w_n$, \textit{J. reine angew. Math.} \textbf{232} (1968), 122--135.

\bibitem{H} D. Harbater, Convergent arithmetic power series, \textit{Amer. J. Math.}, \textbf{106} (4) (1984), 801--846.

\bibitem{diego} J. Huang, D. Marques and M. Mereb, Algebraic values of transcendental functions at algebraic points. {\it Bull. Austral. Math. Soc.} {\bf 82} (2010), 322--327.

\bibitem{bookmahler} K. Mahler, \textit{Lectures on Transcendental Numbers}, Lecture Notes in Math., \textbf{546}, Springer-Verlag, Berlin, 1976.

\bibitem{M} K. Mahler, Arithmetic properties of lacunary power series with integral coefficients,
\textit{J. Austral. Math. Soc.} \textbf{5} (1965), 56--64.

\bibitem{MR} D. Marques, J. Ramirez, On exceptional sets: the solution of a problem posed by K. Mahler, \textit{Bull. Austral. Math. Soc}., \textbf{94} (2016), 15--19.

\bibitem{DG} D. Marques, C. G. Moreira, A positive answer for a question proposed by K. Mahler, to appear in {\it Math. Ann.}


\bibitem{19} P. St\"{a}ckel, Ueber arithmetische Eingenschaften analytischer Functionen, \textit{Math. Ann.} \textbf{46} (1895), 513--520.


\bibitem{tubbs}  R. Tubbs. \textit{Hilbert’s seventh problem}, volume 2 of Institute of Mathematical Sciences Lecture Notes. Hindustan Book Agency, New Delhi, 2016.


\bibitem{wal1} M. Waldschmidt, Algebraic values of analytic functions, Proceedings of the International Conference on Special Functions and their Applications (Chennai, 2002).  \textit{J. Comput. Appl. Math.} \textbf{160} (2003), 323--333.

\bibitem{wal2} M. Waldschmidt, \textit{Diophantine Approximation on Linear Algebraic Groups}. Springer-Verlag, Berlin Heidelberg, New York (2000).




\end{thebibliography}
\end{document}